\documentclass[11pt]{article}

\usepackage{amsthm,amsmath,amsmath,amssymb,amsthm,amsfonts,amstext,amsbsy,amscd}

\RequirePackage[colorlinks=true,citecolor=blue,urlcolor=blue]{hyperref}

\RequirePackage[OT1]{fontenc}

\usepackage{mathrsfs, amscd, amsfonts}

\let\epsilon=\varepsilon
\let\ep=\epsilon
\let\phi=\varphi
\let\al=\alpha \let\be=\beta \let\Ga=\Gamma \let\De=\Delta
\let\de=\delta
  \let\ga=\gamma
\let\Om=\Omega  
\let\si=\sigma  
\let\la=\lambda  
\let\tilde=\widetilde

\newcommand{\abs}[1]{\lvert#1\rvert}

\newcommand{\bigpar}[1]{\bigl(#1\bigr)}
\newcommand{\Bigpar}[1]{\Bigl(#1\Bigr)}

\newcommand{\bigcro}[1]{\bigl[#1\bigr]}

\newcommand{\norm}[1]{\|#1\|}
\newcommand{\bignorm}[1]{\bigl\|#1\bigr\|}

\newcommand{\PE}[1]{{\lfloor#1\rfloor}}

\newcommand{\field}[1]{\mathbb{#1}}
\newcommand{\R}{\field{R}}

\newcommand{\F}{{\mathscr{F}}}

\def\der^#1_#2{\frac{\partial^{#1}}{\partial {#2}^{#1}}}
\newcommand{\ind}[1]{\mathbf 1_{#1}}
\newcommand{\set}[1]{\left\{#1\right\}}

\theoremstyle{plain}
\newtheorem{theorem}{Theorem} \newtheorem{corollary}{Corollary}
\newtheorem{lemma}{Lemma}
\newtheorem{prop}{Proposition}
{\bf}{\rm}%
\theoremstyle{definition}
\theoremstyle{remark}
\newtheorem{remark}{Remark}
\newtheorem{definition}{Definition}

{\smallskip}%

\newcommand{\E}{\field{E}}

\newcommand{\T}{{\!\top}}

\newcommand{\PHI}{\boldsymbol{\phi}}
\newcommand{\PSI}{\boldsymbol{\psi}}

\begin{document}


\title{Scaling limits for Hawkes processes and application to financial statistics}


\author{E. Bacry\footnote{CMAP CNRS-UMR 7641 and \' Ecole Polytechnique, 91128 Palaiseau, France},\;S. Delattre\footnote{Universit\'e Paris Diderot and LPMA CNRS-UMR 7599, Bo\^ite courrier 7012
75251 Paris, France},\; M. Hoffmann\footnote{ENSAE-CREST and LAMA CNRS-UMR 8050, 3, avenue Pierre Larousse, 92245 Malakoff, France}\; and J.F. Muzy\footnote{SPE CNRS-UMR 6134 and Universit\'e de Corte, 20250 Corte, France}}

%
%



\date{}
\maketitle

\begin{abstract}
We prove a law of large numbers and a functional central limit theorem for multivariate Hawkes processes observed over a time interval $[0,T]$ in the limit $T \rightarrow \infty$. We further exhibit the asymptotic behaviour of the covariation of the increments of the components of a multivariate Hawkes process, when the observations are imposed by a discrete scheme with mesh $\Delta$ over $[0,T]$ up to some further time shift $\tau$. The behaviour of this functional depends on the relative size of $\Delta$ and $\tau$ with respect to $T$ and enables to give a full account of the second-order structure. As an application, we develop our results in the context of financial statistics. We introduced in \cite{BDHM1} a microscopic stochastic model for the variations of a multivariate financial asset, based on Hawkes processes and that is confined to live on a tick grid. We derive and characterise the exact macroscopic diffusion limit of this model and show in particular its ability to reproduce important empirical stylised fact such as the Epps effect and the lead-lag effect. Moreover, our approach enable to track these effects across scales in rigorous mathematical terms. 

\end{abstract}







\section{Introduction}

\subsection{Motivation and setting}

Point processes have long served as a representative model for event-time based stochastic phenomena that evolve in continuous time. A comprehensive mathematical development of the theory of point processes can be found in the celebrated textbook of Daley and Vere-Jones \cite{DVJ1}, see also the references therein. In this context, mutually exciting processes form a specific but quite important class of point processes that are mathematically tractable and widely used in practice. They were first described by Hawkes in 1971 \cite{Hawkes71, Hawkes71bis} and according to \cite{DVJ1}: ``Hawkes processes figure widely in applications of point processes to seismology, neurophysiology, epidemiology, and reliability... One reason for their versatility and popularity is that they combine in one model both a cluster process representation and a simple conditional intensity representation, which is moreover linear. It comes closest to fulfilling for point processes, the kind of role that the autoregressive model plays for conventional time series." Informally, to a multivariate $d$-dimensional counting process $N = (N_1,\ldots, N_d)$ with values in $\mathbb{N}^d$ is associated an intensity function $(\lambda_1,\ldots, \lambda_d)$ defined by
$$P\big(N_{i}\;\text{has a jump in}\;[t,t+dt]\;\big|\;\F_t\big)=\lambda_{i,t}\,dt,\;\;i=1,\ldots, d$$
where $P$ stands for probability and $\F_t$ is the sigma-field generated by $N$ up to present time $t$. A multivariate Hawkes process has intensity
\begin{equation} \label{intensiteHawkes}
 \la_{i,t} = \mu_{i} + \int_{(0,t)} \sum_{j=1}^d \phi_{ij}(t-s) dN_{j,s},\;i=1,\ldots,d
\end{equation}
and is specified by $\mu_i \in \R_+ = [0,\infty)$ and for $i=1,\ldots, d$, the $\varphi_{ij}$ are functions from $\R_+$ to $\R_+$. More in Section \ref{def Hawkes} below for rigorous definitions.

The properties of Hawkes processes are fairly well known: from a probabilistic point a view, the aforementioned book of Daley and Vere-Jones \cite{DVJ1} gives a concise synthesis of earlier results published by Hawkes \cite{Hawkes71, Hawkes71bis, Hawkes72} that focus on spectral analysis following Bartlett \cite{Bartlett} and cluster representation, see Hawkes and Oakes \cite{HO74}. From a statistical perspective, Ogata studied in \cite{Ogata78} the maximum likelihood estimator and showed that parametric experiments generated by the observation of Hawkes processes are regular in the sense of Fisher information (see for instance \cite{VdV} for a modern formulation of statistical regularity). Ogata also studied the numerical issue of exact simulation of Hawkes processes \cite{Ogata81}, and this topic has known several developments since, see \cite{Moller} for a review. Recently, the nonparametric estimation of the intensity functions has been investigated by Reynaud-Bouret and Schbath \cite{RBS} and Al Dayri {\it et al.} \cite{ABM} in dimension $d=1$.\\ 

However, in all these papers and the references therein, the focus is on the  ``microscopic properties" of Hawkes processes, {\it i.e.} their infinitesimal evolution, possibly under a stationary regime or for a large time horizon $[0,T]$ in order to guarantee a large number of jumps for statistical inference purposes. In the present work, we are rather interested in the ``macroscopic properties"  of Hawkes processes, in the sense of obtaining a limit behaviour for the multivariate process $(N_{Tv})_{v \in [0,1]}$ as $T\rightarrow \infty$, for a suitable normalisation. Our interest originates in financial data modelling: in \cite{BDHM1}, we introduced a stochastic model for the price $S= (S_1,\ldots, S_n)$ of a multivariate asset, based on a $d=2n$ dimensional Hawkes process of the form \eqref{intensiteHawkes}, with representation
\begin{equation} \label{rep prix}
S_1=N_1-N_2,\ S_2 = N_3-N_4,\ldots,\ S_n = N_{d-1}-N_d.
\end{equation}
In this context, the  fact that the $S_i$ take values on  $\mathbb{Z}$ accounts for the discretness of the price formation under a limit order book.   
If we take $\mu_{2i-1}=\mu_{2i}$ for every $1 \leq i \leq d$, the process $S$ is centred and the mutually exciting properties of the intensity processes $\lambda_i$ under \eqref{intensiteHawkes} allow to reproduce empirically microstructure noise and the Epps effect, as demonstrated in \cite{BDHM1}. Microstructure noise -- a major stylised fact in high frequency financial data (see {\it e.g.} \cite{ABDL, AMZ, AMZbis, RobRos, Ros, barn}) -- is characterised by the property that in microscopic scales\footnote{when the data are sampled every few seconds or less.}, an upward change of price for the component $S_i$ is more likely to be followed by a downward jumps of $S_i$ and {\it vice versa}. Similarly, a jump of a component $N_{2i-1}$ or $N_{2i}$ of  $S_i$ at time $t$ will increase the conditional probability, given the history up to time $t$, that a jump occurs for the component $N_{2j-1}$ or $N_{2j}$ of  $S_{j}$ through the exciting effect of the kernels $\varphi_{2i-1,2j-1}$, $\varphi_{2i-1,2j}$, $\varphi_{2i,2j-1}$ or $\varphi_{2i,2i}$, thus creating a certain dependence structure and the Epps effect, another well-document stylised fact implying that the empirical correlation bewteen two assets vanishes on microscopic scales \cite{Epps}. However, both microstructure noise and the Epps effect vanish at coarser scales where a standard diffusion regime dominates. Thus, a certain macroscopic stability property is desirable, namely the property that $S$ behaves like a continuous diffusion on coarse scales.  Also, one would like to be able to track the aforementioned microscopic effects produced by the kernels $\varphi_{ij}$ in the diffusion limit of $S$. This is the main topic of the paper and also its novelty, as far as the modelling angle is concerned.\\ 

The importance of mutually exciting processes has expanded in financial econometrics over the last years, see \cite{ASCDL, BH, ELL, Hewlett} among others. From a wider perspective however, a mathematical analysis of the limiting behaviour of Hawkes processes has a relevance beyond finance: apart from the interest of such a study for its own sake, one could also presumably apply asymptotic results in other application fields. These include  seismology of course, for which Hawkes processes were originally introduced  (see \cite{ZOVJ} for instance) and also the citation we used above from the monograph \cite{DVJ1}; in a more speculative or prospective way, one could mention recent related areas such as traffic networks \cite{BM} or genomic analysis \cite{RBS} where mutually exciting processes take a growing importance. 

\subsection{Main results and organisation of the paper}

Because of this versatility in the use of Hawkes processes, we present our results in a fairly general setting in Sections \ref{def Hawkes}, \ref{Brownian limits} and \ref{signature plot} with relatively minimal assumptions. It is only in the subsequent Section \ref{applications} that we consistently apply our results to the multivariate price model $S$ described by \eqref{intensiteHawkes}-\eqref{rep prix} for financial data modelling.\\

In Section \ref{def Hawkes}, we recall the construction of Hawkes processes in a rigourous and general setting. We state in particular the fact that as soon as the kernels $\varphi_{ij}$ that define the intensity process $\lambda$ are locally integrable, the existence of $N$ with uniqueness in law is guaranteed on a rich enough space, as follows from the theory of predictable projection of integer-valued random measures of Jacod \cite{JJ75}. We state in Section \ref{Brownian limits} a law of large numbers (Theorem \ref{LG}) and a functional central limit theorem (Theorem \ref{TCL}). The law of large numbers takes the form
\begin{equation} \label{pres LLN}
\sup_{v \in [0,1]}\big\|T^{-1} N_{Tv} - v\,({\bf Id-K})^{-1} \mu\big\|\rightarrow 0 \;\; \text{as}\;\;T \rightarrow \infty
\end{equation}
almost surely and in $L^2(P)$. We set $\|\cdot\|$ for the Euclidean norm on $\R^d$. The limit is described by $\mu=(\mu_1,\ldots, \mu_d)$ identified with a column vector and the $d\times d$ matrix ${\bf K}=\int_0^\infty \PHI(t)\,dt$, where $\PHI=(\varphi_{ij})$ is a matrix of functions from $\R_+$ to $\R_+$. The convergence \eqref{pres LLN} holds under the assumption that the spectral radius of ${\bf K}$ is strictly less than $1$, which guarantees in particular all the desirable integrability properties for $N$. As for a central limit theorem, under the additional condition that $t^{1/2}\PHI(t)$ is integrable componentwise, a consequence of Theorem \ref{TCL} is the convergence of the processes
$$ \sqrt{T} \Bigpar{T^{-1}N_{Tv}-  v(\mathbf{Id-K})^{-1}\mu }, \quad v\in[0,1]$$
to
$$ (\mathbf{Id-K})^{-1} \mathbf{\Sigma}^{1/2} W_v , \quad v\in[0,1].$$
in law for the Skorokod topology. The limit is described by a standard $d$-dimensional Brownian motion $(W_v)_{v \in [0,1]}$ and the  $d\times d$ diagonal matrix $\mathbf{\Sigma}$ defined by $\mathbf{\Sigma}_{ii}= \big(\mathbf{(Id-K)^{-1}} \mu\big)_i$. The proof of Theorems \ref{LG} and \ref{TCL} relies on the intensity theory of point processes: using that $N_t-\int_0^t \lambda_s\, ds$ is a $d$-dimensional martingale, if we set $X_t=N_t-E(N_t)$ with $E(\cdot)$ denoting expectation, we then have the fundamental representation
$$X_t= M_t + \int_0^t \PHI(t-s) X_s \,ds$$
of $X$ as a sum of a martingale and a convolution product of $\PHI$ with $X$ itself. Separating the two components, we can then take advantage of the powerful theory of limit theorems of semimartingales, as exposed in the comprehensive book of Jacod and Shiryaev \cite{JJetAS}.\\
 
In Section \ref{signature plot}, we consider yet another angle of study that is useful for applications. Given two square integrable point processes $N$ and $N'$ with value in $\mathbb{N}^d$, setting $X_t = N_t - E(N_t)$ and $X'_t = N'_t-E(N'_t)$, one can define the empirical cross-correlation of $N$ and $N'$ at scale $\Delta$ over $[0,T]$ as
$$\mathbf{V}_{\De,T}(N,N') = \frac{1}{T} \sum_{i=1}^{\PE{T/\De}} \bigpar{X_{i\De}-X_{(i-1)\De}} \big(X'_{i\De}-X'_{(i-1)\De}\big)^\T,$$
where $X_t$ and $X'_t$ are identified as column vectors and $(\cdot)^\T$ denotes transposition in $\R^d$. In rigorous terms,  $\mathbf{V}_{\De,T}(N,N')$ shall rather be called the empirical correlation matrix of the increments of $X$ and $X'$ sampled at scale $\Delta$ over $[0,T]$. This object is of major importance in practice: viewed as a function of $\Delta$, it reveals the correlation structure across scales between $N$ and $N'$. It is crucial for the understanding of the transition of $(N_{Tv})_{v \in [0,1]}$ from a microscopic regime of a point process to the macroscopic behaviour of a Brownian diffusion, up to appropriate normalisation, as $T \rightarrow \infty$. More precisely, given a time shift $\tau = \tau_T\in \R$, we focus in the paper on the $\tau$-shifted empirical correlation matrix of $N$ at scale $\Delta$, namely\footnote{This is actually a more general object since we can formally recover $\mathbf{V}_{\De,T}(N,N')$ from $\mathbf{V}_{\De,T}(N'',N''_{\tau+\cdot})$, where $N''=(N,N')$ is a $2d$-dimensional point process by letting $\tau \rightarrow 0$ in $\mathbf{V}_{\De,T}(N^{''},N^{''}_{\tau+\cdot})$.} $\mathbf{V}_{\De,T}(N,N_{\tau+\cdot})$. We prove in Theorem \ref{SP} that $\mathbf{V}_{\De,T}(N,N_{\tau+\cdot})$ is close in $L^2(P)$ to a deterministic explicit counterpart 
$$ \mathbf{v}_{\De,\tau}=\int_{\mathbb{R}_+^2} \bigpar{1-\tfrac{|t-s-\tau|}{\De}}^+
\bigpar{\mathbf{Id}\,\de_0(ds) + \PSI(s)ds} \mathbf{\Sigma} \bigpar{\mathbf{Id}\,\de_0(dt) + \PSI(t)^\T dt}$$
under appropriate conditions on $\tau$ and $\Delta$ relative to $T$ as $T \rightarrow \infty$. The approximation $\mathbf{v}_{\De,\tau}$ is described by $\mathbf{\Sigma}$ that already appears in Theorems \ref{LG} and \ref{TCL} and $ \PSI = \sum_{n\ge 1} \PHI_n$, where $\PHI_n$ is the matrix of $n$-fold convolution product of $\PHI$. This apparently cumbersome formula paves the way to explicit computations that yield crucial information about the dependence structure of functionals of $N$, as later developed in the paper.\\

Section \ref{applications} is devoted to some applications of Theorems \ref{LG}, \ref{TCL} and \ref{SP} to the study of the multivariate price model $S$ described by \eqref{intensiteHawkes}-\eqref{rep prix} for financial data modelling. We elaborate on particular examples the insight given by $\mathbf{v}_{\De,\tau}$. Concerning macroscopic limits of a univariate price $S=N_1-N_2$ obtained for $d=2$, we derive in Proposition \ref{limitemacrodim1} the limiting variance of the rescaled process $T^{-1/2}(N_{1,Tv}-N_{2,Tv})$ in the special case where $\varphi_{1,2}=\varphi_{2,1}=\varphi$ and $\varphi_{1,1}=\varphi_{2,2}=0$. This oversimplification enables to focus on the phenomenon of microstructure noise and reveals the macroscopic trace of $\varphi$ in the variance of the limiting Brownian motion. We discuss the theoretical implications of the obtained formulas in statistical finance. 
In the same way, we explicit in Proposition \ref{ex4prop1} the macroscopic correlation structure obtained with a bivariate price process $S=(S_1,S_2)=(N_1-N_2,N_3-N_4)$ when only cross-excitations $\varphi_{1,3} = \varphi_{3,1}$ and $\varphi_{2,4} = \varphi_{2,4}$ are considered, the other components $\varphi_{ij}$ being set to $0$ for simplicity. Again, we obtain an illuminating formula that characterises the influence of the microscopic dynamics on the macroscopic correlation of the Brownian limit. This enables in particular in Proposition \ref{ex4prop2} to demonstrate how Hawkes process can account for the Epps effect \cite{Epps} and also the lead-lag effect between two financial assets that has been given some attention in the literature recently \cite{HRY, AbPo, AH}.\\

Section \ref{preparation} develops some preliminary tools for the proofs. Section \ref{proofLG}, \ref{proofTCL} and \ref{proofSP} are devoted to the proof of Theorems \ref{LG}, \ref{TCL} and \ref{SP} respectively. Some technical results and computations of the application Section \ref{applications} are delayed until an appendix.
\section{Multivariate Hawkes processes} \label{def Hawkes}
Consider a measurable space $(\Om,\F)$ on which is defined a non-decreasing sequence of random variables $(T_n)_{n\ge 1}$  taking their values in $(0,\infty]$, and such that $T_n < T_{n+1}$ on the event $\set{T_n<\infty}$, for all $n\ge 1$. Let $(Z_n)_{n\ge 1}$ be a sequence of discrete random variables taking their values in $\set{1,\dots,d}$ for some positive integer $d$.
Define, for $t\geq 0$
$$ N_{i,t}=\sum_{n\ge 1} \ind{\set{T_n\le t}\cap\set{Z_n=i}}.$$
Remark that $N_{i,0}=0$ by construction. We endow $\Om$ with the filtration $(\F_t)_{t\ge 0}$ where
$\F_t$ is the $\si$-algebra generated by the random variables $N_{i,s}$, $s\le t$, $1\le i\le d$.
According to Jacod \cite{JJ75},  for any progressively measurable non-negative processes
$(\la_{1,t})_{t\ge 0},\dots,(\la_{d,t})_{t\ge 0}$ satisfying
$$ \int_0^{T_n} \la_{i,s}\,ds<\infty\quad \text{almost-surely,}$$
there exists at most one probability measure $P$ on $(\Om, \F_\infty)$ such that the compensator (or predictable projection) of
the integer-valued random measure 
$$ N(dt,dx)=\sum_{n\ge 1} {\bf 1}_{\set{T_n<\infty}}\de_{(T_n,Z_n)}(dt,dx)$$
on $(0,\infty)\times\set{1,\dots,d}$ is
$$\nu(dt,dx)=\sum_{i=1}^d \la_{i,t}\,dt \otimes \de_{i}(dx),$$
where $\delta$ is the Dirac mass. In other words,
for all $n\ge 1$, for all $i\in \set{1,\dots,d}$, the process
$$ N_{i,t\land T_n} - \int_0^{t\land T_n} \la_{i,s}\,ds  $$
is a $(\F_t)$-martingale. This implies that the law of the $d$-dimensional process $(N_1,\dots,N_d)$ is characterised
by $(\la_1,\dots,\la_d)$. Moreover if $\Om$ is rich enough we have the existence of such a probability measure $P$.
\begin{definition} \label{defrigHawkes}
We say that $N=(N_1,\dots,N_d)$ is a multivariate Hawkes process when
\begin{equation} \label{Hawkes}
 \la_{i,t} = \mu_{i} + \int_{(0,t)} \sum_{j=1}^d \phi_{ij}(t-s) dN_{j,s}
 \end{equation}
where $\mu_i\in\R_+$ and $\phi_{i,j}$ is a function from $\R_+$ to $\R_+$.
\end{definition}
We have a non-explosion criterion, as a consequence for instance of equality \eqref{compensateurN} in Lemma \ref{maj}, Section \ref{preparation} below.
\begin{lemma}[Non-explosion criterion]
Set $T_\infty=\lim_n T_n$. Assume that the following holds:
\begin{equation}\label{nonexplosion}
\int_0^t \phi_{ij}(s) ds <\infty
\quad \text{for all $i,j,t$.}
\end{equation}
Then $T_\infty=\infty$ almost surely.
\end{lemma}

\section{Law of large numbers and functional central limit theorem} \label{Brownian limits}
On a rich enough probability space $(\Omega, \F, P)$, we consider a Hawkes process $N=(N_t)_{t \geq 0}$ according to Definition \ref{defrigHawkes}, satisfying \eqref{nonexplosion} and specified by the vector
$$ \mu=(\mu_1,\dots,\mu_d)$$
and the $d\times d$-matrix valued function
$$\PHI=(\phi_{i,j})_{1\le i,j\le d}.$$
Note that in this setting, we do not assume a stationary regime for $N$.
Consider the assumption
\begin{equation*} \label{sta}\tag{\mbox{\bf A1}}
\begin{minipage}{10cm} \em
For all $i,j$ we have $\int_0^\infty \phi_{ij}(t) dt < \infty$ and the spectral radius $\rho({\bf K})$ of the matrix ${\bf K}=\int_0^\infty \PHI(t)\,dt$ satisfies $\rho({\bf K}) <1.$
\end{minipage}
\end{equation*}
First we have a law of large numbers in the following sense:
\begin{theorem} \label{LG}
Assume that \eqref{sta} holds. 
Then $N_t\in L^2(P)$ for all $t\ge 0$ and we have
$$ \sup_{v\in[0,1]}\bignorm{T^{-1}N_{Tv} - v\,({\bf Id-K})^{-1} \mu} \rightarrow 0\;\; \text{as}\;\;T \rightarrow \infty$$
almost-surely and in $L^2(P)$. 
%
\end{theorem}

Next we have an associated functional central-limit theorem. Introduce the functions $\PHI_n$ defined on $\R_+$ and with values in the set of $d\times d$-matrices with entries in $[0,\infty]$  by
\begin{equation} \label{defPHI}
 \PHI_1=\PHI, \quad \PHI_{n+1}(t)=\int_0^t \PHI(t-s) \PHI_n(s)\,ds, \quad n\ge 1.
\end{equation}
Under \eqref{sta} we have
$ \int_0^\infty \PHI_n(t)\,dt = {\mathbf K}^n$ hence 
the series $\sum_{n\ge 1} \PHI_n$ converges in $L^1(dt)$. We set
\begin{equation}\label{defPSI}
\PSI = \sum_{n\ge 1} \PHI_n 
\end{equation}

\begin{theorem} \label{TCL} Assume that \eqref{sta} holds.
We have
$$ E(N_t)= t\mu + \Bigpar{\int_0^t \PSI(t-s) s \, ds}\mu$$
where $\PSI\in L^1(dt)$ is given by \eqref{defPSI}.
Moreover, the processes
$$ \frac{1}{\sqrt{T}} \Bigpar{ N_{Tv}-  E(N_{Tv}) }, \quad v\in[0,1] $$
converge in law for the Skorokod topology to
$$ (\mathbf{Id-K})^{-1} \mathbf{\Sigma}^{1/2} W_v , \quad v\in[0,1]$$
as $T \rightarrow \infty$, where $(W_v)_{v\in[0,1]}$ is a standard $d$-dimensional Brownian motion
and $\mathbf{\Sigma}$ is the diagonal matrix such that $\mathbf{\Sigma}_{ii}= (\mathbf{(Id-K)}^{-1} \mu)_i$.
\end{theorem}
Consider now the following restriction on $\PHI$:
\begin{equation*} \label{uibis} \tag{{\bf A2}}
\int_0^\infty \PHI(t) \, t^{1/2}\,dt < \infty \quad \text{componentwise.}
\end{equation*}
Using Theorem \ref{LG} and Assumption \eqref{uibis}, we may re\-place $T^{-1}E(N_{Tv})$ by its limit in Theorem \ref{TCL} and obtain the following corollary. 
\begin{corollary} \label{TCLbis} Assume that \eqref{sta} and \eqref{uibis} hold. Then
the processes
$$ \sqrt{T} \Bigpar{ \frac{1}{T}N_{Tv}-  v(\mathbf{Id-K})^{-1}\mu }, \quad v\in[0,1] $$
converge in law for the Skorokod topology to
$$ (\mathbf{Id-K})^{-1} \mathbf{\Sigma}^{1/2} W_v , \quad v\in[0,1]$$
as $T \rightarrow \infty$.
\end{corollary}

%
%

\section{Empirical covariation across time scales} \label{signature plot}
For two square integrable counting processes $N$ and $N'$ with values in $\mathbb{N}^d$, set
$$X_t= N_t- E(N_t),\;\;X'_t = N'_t-E(N'_t).$$
The empirical covariation across time scales of $N$ and $N'$ is the process $\mathbf{V}_{\De,T}(N,N')$, $T>0$, $\De>0$, taking values in the set of $d\times d$ matrices and defined as
$$ \mathbf{V}_{\De,T}(N,N') = \frac{1}{T} \sum_{i=1}^{\PE{T/\De}} \bigpar{X_{i\De}-X_{(i-1)\De}} \bigpar{X'_{i\De}-X'_{(i-1)\De}}^\T$$
where $\bigpar{X_{i\De}-X_{(i-1)\De}}$ is identified as a column vector and
$\bigpar{X_{i\De}-X_{(i-1)\De}}^\T$ denotes its transpose, and we set $X_t=0$ for $t\le 0$. More precisely, given a time shift $\tau \in \R$, we are interested in the behaviour of $\mathbf{V}_{\De,T}(N,N_{\tau+\cdot})$. In essence, $\mathbf{V}_{\De,T}(N,N_{\tau+\cdot})$ can be viewed as a multivariate cross-correlogram across scales $\Delta$ of $N$: it can be consistently measured from empirical data and its limiting behaviour as $T \rightarrow \infty$ plays a key tool in understanding the second-order structure of linear functions of $N$ across scales $\Delta$. 
\begin{theorem} \label{SP}
In the same setting as in Theorem \ref{TCL},
let $(\De_T)_{T>0}$ and $(\tau_T)_{T>0}$ be two families of real numbers such that $\De_T>0$.
If $\De_T/T\to 0$ and $\tau_T/T\to 0$ as $T\to\infty$, we have
$$ \mathbf{V}_{\De_T,T}(X,X_{\tau_T+\cdot}) - \mathbf{v}_{\De_T,\tau_T} \to 0\;\;\text{as}\;\;T \rightarrow \infty\;\;\text{in}\;\;L^2(P)$$
where
\begin{multline} \label{signath}
\mathbf{v}_{\De,\tau}=
\bigpar{1-\tfrac{\abs{\tau}}{\De}}^+\, \mathbf{\Sigma} + \int_{\R_+^2} \!\! ds\, dt\,\bigpar{1-\tfrac{\abs{t-s-\tau}}{\De}}^+ \PSI(s) \mathbf{\Sigma}\, \PSI(t)^\T+\\
+ \int_0^\infty ds\,(1-\tfrac{\abs{s+\tau}}{\De})^+ \PSI(s) \mathbf{\Sigma}
+ \int_0^\infty ds\,(1-\tfrac{\abs{s-\tau}}{\De})^+ \mathbf{\Sigma}\,\PSI(s)^\T,
\end{multline}
or equivalently,
$$ \mathbf{v}_{\De,\tau}=\int_{\R_+^2} \bigpar{1-\tfrac{\abs{t-s-\tau}}{\De}}^+
\bigpar{\mathbf{Id}\,\de_0(ds) + \PSI(s)ds} \mathbf{\Sigma} \bigpar{\mathbf{Id}\,\de_0(dt) + \PSI(t)^\T dt},$$
where $\mathbf{\Sigma}$ is the diagonal matrix such that $\mathbf{\Sigma}_{ii}= \big(\mathbf{(Id-K)^{-1}} \mu\big)_i$ and the function $\PSI$ is given by \eqref{defPSI}.
\end{theorem}
\begin{remark}
On can check that $\mathbf{v}_{\De,\tau}=E\bigpar{(N_{\De}-N_0)(N_{\De+\tau}-N_{\tau})^\T}$
where $N$ is a counting process of the (unique in law) stationary multivariate Hawkes process on $\R$ associated to $\mu$ and $\PHI$. Thus, another way to obtain $\mathbf{v}_{\De,\tau}$ is to compute $E\bigpar{(N_{\De}-N_0)(N_{\De+\tau}-N_{\tau})^\T}$ in the stationary regime, as in \cite{Hawkes71, Hawkes71bis, Hawkes72} by means of the Bartlett spectrum of $N$, see \cite{Bartlett}. However, the stationary restriction is superfluous and moreover, only very specific parametric form of $\varphi_{ij}$ like exponential functions enable to carry such computations.

\end{remark}
\begin{remark}
Obviously, we have $\mathbf{v}_{\De,\tau}= \mathbf{v}_{\De,-\tau}^\T$. 
\end{remark}
\begin{remark}
For fixed $\tau \in \R$, we retrieve the macroscopic variance of Theorem \ref{TCL} by letting $\Delta\rightarrow \infty$.
More precisely, we have
$$\mathbf{v}_{\De,\tau} \to \mathbf{\Ga\, \Sigma}\,\mathbf{ \Ga}^\T\;\;\text{as}\;\;\De\to\infty$$
where $\mathbf{\Ga=(Id-K)^{-1}}$, and the effect of $\tau$ vanishes as $\Delta \rightarrow \infty$.
For all $\tau\not=0$ we have 
$$\mathbf{v}_{\De,\tau} \to \mathbf{0}\;\;\text{as}\;\;\De\to 0.$$
This convergence simply expresses the fact that the two processes $N$ and $N_{\tau+\cdot}$ cannot jump at the same time, producing a flat autocorrelogram for sufficiently small sampling mesh $\Delta$.
\end{remark}
In the same way as Corollary \ref{TCLbis} is obtained from Theorem \ref{TCL}, we have the following refinement of Theorem \ref{SP}.
\begin{corollary} \label{SPbis}
Finally, in the same setting as in Theorem \ref{SP},
assume moreover that \eqref{uibis} holds. Define
$$ \tilde X_t= 
\left\{
\begin{array}{lll}
N_t- t (\mathbf{Id-K})^{-1}\mu &\text{if}& t\ge 0,\\
 0 &\text{if} & t< 0.
\end{array}
\right.$$
 We have
$$ \mathbf{V}_{\De_T,T}(\tilde X,\tilde X_{\tau_T+\cdot}) - \mathbf{v}_{\De_T,\tau_T} \to 0\;\;\text{as}\;\;T\rightarrow \infty\;\;\text{in}\;\;L^2(P),$$
where $\mathbf{v}_{\De,\tau}$ is given by \eqref{signath}.
\end{corollary}

%
%
\section{Application to financial statistics}  \label{applications}
\subsection{The macroscopic trace of microstructure noise} Following \cite{BDHM1}, we introduce a univariate price process $S = (S_t)_{t \geq 0}$ by setting
$$S=N_1-N_2,$$
where $(N_1,N_2)$ is a Hawkes process in the case $d=2$, with 
$$
\left\{
\begin{array}{lll}
\lambda_{1,t} & =\displaystyle & \nu+\int_{(0,t)} \varphi(t-s)dN_{2,s}, \\ \\
\lambda_{2,t} & =\displaystyle  & \nu+\int_{(0,t)} \varphi(t-s)dN_{1,s}
\end{array}
\right.
$$
for some $\nu \in \R_+$ and $\varphi:\R_+\rightarrow \R_+$. With our notation, this corresponds to having $\mu=(\nu,\nu)$ and 
\begin{equation*}
\PHI=\begin{pmatrix}
0 & \varphi \\
\varphi & 0  \\
\end{pmatrix}
\end{equation*}
If $\varphi=0$, we find back a compund Poisson process with intensity $\nu$ and symmetric Bernoulli jumps. This corresponds to the simplest model for a random walk in continuous time, constrained to live on a lattice, the tick-grid in financial statistics accounting for the discreteness of price at fin scales. Microstructure noise corresponds to the property that an upward jump of $S$ will be more likely followed by a downward jump and vice versa. This phenomenon lays its roots in microeconomic analysis of price manipulation of agents \cite{ABDL, AMZ, AMZbis, RobRos, Ros, barn}. In our simple phenomenological setting, it will be reproduced by the introduction of the kernel $\varphi$, as empirically demonstrated in \cite{BDHM1}. The question we can now address is the macroscopic stability of the model. Do we retrieve a standard diffusion in the limit $T\rightarrow \infty$ for an approriate scaling of $S$ and how does the effect of $\varphi$ influence the macroscopic volatility?
By Theorems \ref{LG} and \ref{TCL}, we readily obtain an explicit anwser:
\begin{prop}[Macroscopic trace of microstructure noise] \label{limitemacrodim1}
Assume that 
$\|\varphi\|_{L^1} = \int_0^\infty \varphi(t)dt<1$. Then
$$\big(T^{-1/2}S_{Tv}, v\in [0,1]\big) \rightarrow \big(\sigma W_v,v\in [0,1]\big)\;\;\text{as}\;\;T\rightarrow\infty,$$
in law for the Skorokod topology, where $(W_v)_{v \in [0,1]}$ is a standard Brownian motion and
$$\sigma^2 = \frac{2\nu}{(1-\|\varphi\|_{L^1})(1+\|\varphi\|_{L^1})^2}.$$
\end{prop}

\begin{remark} Note that if we take $\varphi=0$, we retrieve the standard convergence of a compound Poisson process with symmetric jump to a Brownian motion. 
\end{remark}
\begin{remark}
By assumption, $0 \leq \|\varphi\|_{L^1}<1$ and a closer inspection of the function 
$$x \leadsto \sigma(\nu,x)^2=\frac{2\nu}{(1-x)(1+x)^{2}}\;\;\text{for}\;\;x\in [0,1)$$ reveals an interesting feature: for small microstructure effect (namely if $x=\|\varphi\|_{L^1}$ less than $1/3$) the effect of microstructure tends to stabilise the macroscopic variance in the sense that  $\sigma(\nu,x)^2 \leq \sigma(\nu,0)^2$, whereas beyond a critical value $x\approx 0.61$, we have $\sigma(\nu,x)^2 \geq \sigma(\nu,0)^2$ and even $\sigma(\nu,x)^2\rightarrow \infty$ as $x\rightarrow 1$.  
\end{remark}

\subsection{Macroscopic correlations for bivariate assets} \label{correlation}
We now turn to a bivariate price model $S=(S_1,S_2)$ obtained from a Hawkes process in dimension $d=4$, of the form
$$(S_1,S_2)=(N_1-N_2, N_3-N_4)$$
with $\mu=(\mu_1,\mu_2,\mu_3,\mu_4)$ such that
$$\mu_1=\mu_2\;\;\text{and}\;\;\mu_3=\mu_4$$
together with
\begin{equation*}
\PHI=\begin{pmatrix}
0 & 0 & h & 0\\
0 & 0 & 0 & h\\
g & 0 & 0 & 0\\
0 & g & 0 & 0.
\end{pmatrix}
\end{equation*}
The upward jumps of $S_1$ excite the intensity of the upward jumps of $S_2$ and, in a symmetric way, the downward jumps of $S_1$ excite the downward jumps of $S_2$ via the kernel $g:\R_+\rightarrow \R_+$. Likewise,  the upward jumps of $S_2$ excite the upward jumps of $S_1$ and the downward jumps of $S_2$ excite the downward jumps of $S_1$ via the kernel $h:\R_+\rightarrow \R_+$.
For simplicity we ignore other cross terms that could produce microstructure noise within the inner jumps of $S_1$ and $S_2$.  This relativeley simple dependence structure at a microscopic level enables to obtain a non-trivial form of the macroscopic correlation of the diffusion limits of $S_1$ and $S_2$.
\begin{prop} \label{ex4prop1}
Assume that $\|h\|_{L^1}\| g\|_{L^1} <1$. The $2$-dimensional processes $T^{-1/2}\bigpar{S_{1,Tv},S_{2,Tv}}_{v\in[0,1]}$
converge in law as $T\to\infty$ for the Skorokod topology to 
$$\begin{pmatrix} X_1\\ X_2\end{pmatrix}= \frac{\sqrt{2}}{(1-\| h\|_{L^1}\| g\|_{L^1})^{3/2}}
\begin{pmatrix}
\nu_1^{1/2}\, W_1 + \nu_2^{1/2}\|h\|_{L^1}\, W_2 \\
\nu_1^{1/2}\|g\|_{L^1} \,W_1 + \nu_2^{1/2}\, W_2
\end{pmatrix}
$$
with
\begin{equation} \label{defnu}
\nu_1={\mu_1+\|h\|_{L^1} \mu_3},\quad
\nu_2={\mu_3+\| g\|_{L^1} \mu_1}.
\end{equation}
and where $(W_1,W_2) = (W_{1,t},W_{2,t})_{t \in [0,1]}$ is a standard Brownian motion.
\end{prop}
The proof is a consequence of Theorems \ref{LG} and \ref{TCL} and is given in appendix.
\begin{remark} The macroscopic correlation between $S_1$ and $S_2$ is thus equal to the cosine of the angle of the two vectors
$$\bigpar{\nu_1^{1/2}, \nu_2^{1/2}\|h\|_{L^1}}\;\;\text{and}\;\;\bigpar{\nu_1^{1/2}\| g\|_{L^1}, \nu_2^{1/2}}.$$
Obviously, it is always nonnegative and strictly less than $1$ since the determinant $\nu_1^{1/2}\nu_2^{1/2}\bigpar{1-\|h\|_{L^1}\|g\|_{L^1}}$ of the two above vectors is positive unless the $\mu_i$ are all $0$.
\end{remark}
\subsection{Lead-lag and Epps effect through the cross correlations across-scales}
We keep up with the model and the notation of Section \ref{correlation} but we now study the quantities
$$V_{\De,T}(S_1,S_{1,\tau+\cdot}), \ V_{\De,T}(S_2,S_{2,\tau+\cdot}),\ V_{\De,T}(S_1,S_{2,\tau+\cdot}).$$
In particular, the quantity $V_{\De,T}(S_1,S_{2,\tau+\cdot})$ is a powerful tool for the statistical study of lead-lag effects, {\it i.e.} the fact that the jumps of $S_1$ can anticipate on those of $S_2$ and vice-versa, see for instance \cite{AH, HRY}. The Epps effect -- {\it i.e.} the stylised fact statement that the correlation between the increments of two assets vanishes at fine scales -- can be tracked down likewise. Theorem \ref{SP} enables to characterise in principle the limiting behaviour of these functionals. This is described in details in Proposition \ref{ex4prop2} below.\\

We consider $g$ and $h$ as functions defined on $\R$ by setting $g(t)=h(t)=0$ is $t<0$. Assume that $\| h \|_{L^1}\| g\|_{L^1} <1$. Then the series
$$F:=\sum_{n\ge 1} (h\star g)^{\star n}$$
converges in $L^1(\R,dt)$. If $f$ is a function on $\R$ we define $\check f$ by $\check f(t)=f(-t)$. We have
\begin{prop} \label{ex4prop2} Assume that $\| h \|_{L^1}\| g\|_{L^1} <1$. Let $(\De_T)_{T>0}$ and $(\tau_T)_{T>0}$ be two families of real numbers such that $\De_T>0$.
If $\De_T/T\to 0$ and $\tau_T/T\to 0$ as $T\to\infty$ we have
$$ V_{\De_T,T}(S_1,S_{1,\tau_T+\cdot}) -
 C_{11}(\De_T,\tau_T)
 \to 0\;\;\text{as}\;\;T\rightarrow \infty\;\;\text{in}\;\;L^2(P),$$
where
$$ C_{11}(\De,\tau)=\tfrac{2}{1-\|h\|_{L^1}\|g\|_{L^1}}\ga_{\De}\star \bigpar{\de_0+F+\check F+F\star \check{F}} \star
 \bigpar{\nu_1\de_0 + \nu_2•, h\star\check h}(\tau),$$
with $$ \gamma_\De(x)=\bigpar{1-\abs{x}/\De}^+.$$
We also have 
$$ V_{\De_T,T}(S_1,S_{2,\tau_T+\cdot}) -
C_{12}(\De_T,\tau_T)
 \to 0\;\;\text{as}\;\;T\rightarrow\infty\;\;\text{in}\;\;L^2(P)$$
 and
$$ V_{\De_T,T}(S_2,S_{1,\tau_T+\cdot}) -
C_{12}(\De_T,-\tau_T)
\to 0\;\;\text{as}\;\;T\rightarrow\infty\;\;\text{in}\;\;L^2(P),$$
with $$ C_{12}(\De,\tau)=
 \tfrac{2}{1-\|h\|_{L^1}\| g\|_{L^1}}\ga_{\De}\star \bigpar{\de_0+F+\check F+F\star \check{F}} \star \bigpar{\nu_2 \check h + \nu_1 g}(\tau).$$
\end{prop}

\begin{remark}[The Epps effect]
For $f\in L^1$ we have $\ga_\De\star f\to 0$ pointwise as $\De\to 0$. Therefore we obtain 
$$C_{12}(\De,\tau)\to 0\;\;\text{as}\;\;\De\to 0\;\;\text{for every}\;\;\tau\in\R$$
and this characterises the Epps effect. The same argument, together with $\ga_\De(0)=1$, yields the convergence
$$C_{11}(\De,\tau) \to \frac{2\nu_1}{1-\|h\|_{L^1}\|g\|_{L^1}} \ind{\set{\tau=0}}\quad
 \text{as $\De\to 0$.}$$
\end{remark}
 \begin{remark}[The Lead-Lag effect] Following \cite{BDHM1} and as a consensus in the literature, we say that a lead-lag effect is present between $S_1$ and $S_2$  if there exists $\De>0$ and $\tau\not=0$ such that
$$C_{12}(\De,\tau)\not=C_{12}(\De,-\tau).$$
Therefore, an absence of any lead-lag effect is obtained if and only if the function
$\bigpar{\de_0+F+\check F+F\star \check{F}} \star \bigpar{\nu_2 \check h + \nu_1 g}$ is even. This is the case if $\nu_1 g=\nu_2 h$.
Now, let $\ep>0$. If $g=h\star\de_\ep$ and $\mu_1=\mu_3$, then
$$C_{12}(\De,-\tau)=C_{12}(\De,\tau+\ep)\;\;\text{for every}\;\;\De>0,\tau\in\R.$$
This particular choice for $h$ and $g$ models in particular the property that $S_1$ acts on $S_2$ in the same manner as $S_2$ acts on $S_1$ with an extra temporal shift of $\ep$. Since we always have
$$\lim_{\tau\to\pm\infty}C_{12}(\De,\tau)=0,$$
there exists $\tau_0$ such that $C_{12}(\De,\tau_0+\ep)\not=C_{12}(\De,\tau_0)$, or in other words, we have a lead-lag effect.
\end{remark}
\begin{remark}[Macroscopic correlations] Since $\gamma_\De\to 1$ as $\De\to\infty$, we obtain the convergence
\begin{multline*}
C_{11}(\De,\tau) \to
  \tfrac{2}{1-\| h\|_{L^1}\| g\|_{L^1}} \int \bigpar{\de_0+F+\check F+F\star \check{F}} \star
 \bigpar{\nu_1\de_0 + \nu_2 h\star\check h}\\
= \tfrac{2}{1-\| h\|_{L^1}\| g\|_{L^1}}  \Bigpar{1+2\smallint F+(\smallint F)^2}\bigpar{\nu_1+\nu_2(\smallint h)^2}
 =\text{Var}(X_1).
 \end{multline*}
Likewise, we have $C_{12}(\De,\tau)\to \text{Cov}(X_1,X_2)$ as $\Delta \rightarrow \infty$.
\end{remark}
\begin{remark}
Finally, note that if we use the convenient parametrisation $h(t)=\al_1 \exp(-\be_1 t)$ and $g(t)=\al_2\exp(-\be_2 t)$, assuming further $\|h\|_{L^1}\| g\|_{L^1}=\al_1\al_2/(\be_1\be_2)<1$, then standard computations yield the explicit form
$$ F(t)=\frac{\al_1\al_2}{\rho_1-\rho_2}\bigpar{\exp(-\rho_2 t)-\exp(-\rho_1 t)} \ind{\R_+}(t)$$
with $$\rho_1=\frac{1}{2}\Bigpar{\be_1+\be_2+\sqrt{(\be_1-\be_2)^2+4\al_1\al_2}},$$
$$\rho_2=\frac{1}{2}\Bigpar{\be_1+\be_2-\sqrt{(\be_1-\be_2)^2+4\al_1\al_2}}$$
and we have $0 < \rho_2 < \rho_1$.
This allows to obtain a close formula for $C_{11}$ and $C_{12}$. We do not pursue these computations here.
\end{remark}
\section{Preparation for the proofs} \label{preparation}
In the sequel, we work in the setting of Sections \ref{def Hawkes} and \ref{Brownian limits} under Assumption \eqref{sta}.
\begin{lemma} \label{maj} For all finite stopping time $S$ one has:
\begin{gather} \label{compensateurN}
E(N_S)= \mu E(S) + E\Bigpar{\int_0^S \PHI(S-t) N_t dt}\\
\label{majEN}
E(N_S) \le (\mathbf{Id-K})^{-1} \mu \, E(S) \quad\text{componentwise}.
\end{gather}
\end{lemma}
\begin{proof}
Recall that $(T_p)_{p\ge 1}$ denote the successive jump times of $N$ and set $S_p=S\land T_p$.
Since the stochastic intensities $\la_i$ are given by \eqref{Hawkes} one has
$$E(N_{S_p})=\mu E(S_p)+ E\Bigpar{\int_0^{S_p} dt \int_{(0,t)} \PHI(t-s) dN_s}.$$
Moreover, by Fubini theorem
\begin{align*}
\int_0^{S_p} dt \int_{(0,t)} \PHI(t-s) dN_s&= \int_{[0,S_p)} \bigpar{ \int_s^{S_p} \PHI(t-s) dt} dN_s\\
&= \int_{[0,S_p)}  \bigpar{\int_0^{S_p-s} \PHI(t) dt} dN_s.
\end{align*}
Now, integrating by part with $\mathbf{\Phi}(t)=\int_0^t \PHI(s) ds$, we can write
$$0= \mathbf{\Phi}(0)  N_{S_p} - \mathbf{\Phi}(S_p)  N_{0}
= \int_{(0,S_p]} \mathbf{\Phi}(S_p-t) dN_t - \int_0^{S_p} \PHI(S_p-t) N_t dt.$$
Remark that both sides of the above equality are finite since $\sum_{i=1}^d N_{i,S_p} \le p$.
We obtain
$$ E(N_{S_p})=\mu E(S_p)+ E\Bigpar{\int_0^{S_p} \PHI(S_p-t) N_t dt}$$
and derive \eqref{compensateurN} using that $N_{S_p} \uparrow N_S$ as $p\to\infty$ and
$$  \int_0^{S_p}\! \PHI(S_p-t) N_t dt= \int_0^{S_p}\! \PHI(t) N_{S_p-t} dt
\, \uparrow \, \int_0^{S}\! \PHI(t) N_{S-t} dt= \int_0^{S}\! \PHI(S-t) N_t dt .$$
We next prove \eqref{majEN}.
We have
\begin{align*}
E( N_{S_p}) &= E\Bigpar{S_p\mu + \int_0^{S_p} \PHI(S_p-t) N_t \,dt } \\
&\le E(S_p)\mu + E\Bigpar{ \int_0^{\infty} \PHI(S_p-t) N_t \,dt } \quad \text{componentwise,}\\
& = E(S_p)\mu + \mathbf{K} E(N_{S_p}).
\end{align*}
By induction 
$$E( N_{S_p}) \le \Bigpar{\mathbf{Id} + \mathbf{K} + \dots + {\mathbf{K}}^{n-1}} {E(S_p) \mu } +
{\mathbf{K}}^n E(N_{S_p})$$
componentwise for all integer $n$.
On the one hand, $\sum_{i=1}^d N_{i,S_p} \le p$. On the other hand, since $\rho(\mathbf{K})<1$ we have $\mathbf{K}^n \to 0$ as $n\to\infty$ and
$$\sum_{n=0}^\infty \mathbf{K}^n=(\mathbf{Id-K})^{-1},$$
therefore
$$E( N_{S_p}) \le (\mathbf{Id-K})^{-1} E(S)\mu.$$
This readily yields \eqref{majEN} since $E(N_S)=\lim_p E(N_{S_p})$.
\end{proof}

Let $\PHI_n$, $n\ge 1$, and $\PSI=\sum_{n\ge 1}\PHI_n$ be defined as in Theorem \ref{TCL}. By induction it is easily shown that
$\int_0^\infty \PHI_n(t) dt = \mathbf{K}^n$ for all $n$. Therefore $\int_0^\infty \PSI(t) dt = \sum_{n\ge 1} \mathbf{K}^n$ is finite componentwise by assumption \eqref{sta}. We next state a multivariate version of the well known renewal equation, which proof we recall for sake of completeness. 
\begin{lemma} \label{inverse}
Let $h$ be a Borel and locally bounded function from $\R_+$ to $\R^d$.
Then there exists a unique locally bounded function $f:\R_+\to\R^d$ solution to
\begin{equation} \label{equ}
 f(t)= h(t) + \int_0^t \PHI(t-s) f(s) ds \quad \forall t\ge 0,
 \end{equation}
given by
$$ f_h(t) = h(t) + \int_0^t \PSI(t-s) h(s) ds.$$
\end{lemma}

\begin{proof}
Since $\PSI \in L^1(dt)$ and $h$ is locally bounded, the function $f_h$ is locally bounded. Moreover $f_h$ satisfies \eqref{equ}. It follows that
\begin{align*}
\int_0^t \PHI(t-s) f_h(s) ds&= \int_0^t \PHI(t-s) h(s) ds + \int_0^t ds \PHI(t-s) \int_0^s \PSI(s-r) h(r) dr\\
&=\int_0^t \PHI(t-s) h(s) ds + \int_0^t dr \int_r^t ds \PHI(t-s) \PSI(s-r) h(r)\\
&=\int_0^t \PSI(t-r) h(r)dr
\end{align*}
since
$ \int_0^t \PHI(t-s) \PSI(s) ds = \PSI(t) -\PHI(t)$. As for the uniqueness,  if $f$ satisfies \eqref{equ} then
$$ f_h(t)-f(t) =\int_0^t \PHI(t-s)(f_h(s)-f(s))ds$$
thus if $g_i(t)=\abs{f_{h,i}(t)-f_i(t)}$, $1\le i\le d$, one has
$$ g(t) \le \int_0^t \PHI(t-s) g(s) ds \quad \text{componentwise,}$$
which yields
$$ \int_0^\infty g(t) dt \le \mathbf{K} \int_0^\infty g(t) dt\quad \text{componentwise}.$$
Since $\rho(\mathbf{K})<1$ it follows that $f=f_h$ almost everywhere.
Therefore 
$$\int_0^t \PHI(t-s) f(s) ds =  \int_0^t \PHI(t-s) f_h(s)ds\;\;\text{for all}\;\;t$$
and thus $f=f_h$ since both function satisfies \eqref{equ}.
\end{proof}

Define the $d$-dimensional martingale $(M_t)_{t\ge 0}$ by
$$ M_t= N_t-\int_0^t \la_s\, ds \quad \text{with} \quad \la=(\la_1,\dots,\la_d).$$
\begin{lemma} \label{repre} For all $t\ge 0$:
\begin{gather} \label{repreEN}
E(N_t)= t\mu + \bigpar{\int_0^t \PSI(t-s) s\, ds} \mu,\\
\label{repreN}
N_t- E(N_t)= M_t + \int_0^t \PSI(t-s) M_s ds.
\end{gather}
\end{lemma}
\begin{proof}
By \eqref{compensateurN} of Lemma \ref{maj} and Fubini theorem, we get
$$ E(N_t)= t\mu +\int_0^t \PHI(t-s) E(N_s) \,ds .$$
Besides, $t\leadsto E(N_t)$ is locally bounded in view of \eqref{majEN}.
Applying Lemma \ref{inverse} we obtain \eqref{repreEN}.
The second formula follows from Lemma \ref{inverse} and the fact that, if $X_t=N_t-E(N_t)$,
representation \eqref{repreEN} entails 
$$ X_t= M_t + \int_0^t \PHI(t-s) X_s \,ds.$$
\end{proof}
\section{Proof of Theorem \ref{LG}} \label{proofLG}
\begin{lemma} \label{EN}
Let $p\in[0,1)$ and assume that $\int_0^\infty t^p \PHI(t) \, dt <\infty$ component\-wise. Then
\begin{enumerate}
\item If $p<1$, we have
$$ T^p\Bigpar{T^{-1}E(N_{Tv}) - v(\mathbf{Id-K})^{-1} \mu  } 
\rightarrow 0\;\;\text{as}\;\;T\rightarrow \infty$$
 uniformly in $v\in[0,1]$.
\item If $p=1$, we have
\begin{align*}
& T \Bigpar{ \frac{1}{T} E(N_{T}) - (\mathbf{Id-K})^{-1} \mu  } \\
\rightarrow & -(\mathbf{Id-K})^{-1}\big( \int_0^\infty t\PHI(t)\,dt  \big)(\mathbf{Id-K})^{-1} \mu\;\;\text{as}\;\;T \rightarrow \infty.
\end{align*}
\end{enumerate}
\end{lemma}
\begin{proof}
Let $p\in[0,1]$ and assume that $\int_0^\infty t^p \PHI(t) \, dt <\infty$ componentwise.
We first prove that $\int_0^\infty t^p \PSI(t) \, dt <\infty$ componentwise. For $n\ge 1$, setting $\mathbf{A}_n=\int_0^\infty t^p \PHI_{n}(t)\,dt$, we can write
\begin{align*}
\mathbf{A}_{n+1} &= \int_0^\infty t^p \Bigpar{\int_0^t \PHI(t-s)\PHI_n(s)\,ds}dt\\
&=\int_0^\infty \Bigpar{\int_0^\infty (t+s)^p \PHI(t) \,dt} \PHI_n(s)\,ds\\
&\le \int_0^\infty t^p \PHI(t)\,dt \,\mathbf{K}^n + \mathbf{K}  \int_0^\infty s^p \PHI_{n}(s)\,ds
\quad \text{ with equality if $p=1$,}\\
&= \mathbf{A}_1 \mathbf{K}^n + \mathbf{K A}_n.
\end{align*}
Therefore for all integer $N$,
$$ \sum_{n=1}^N \mathbf{A}_n \le \mathbf{A}_1 + \mathbf{A}_1 \sum_{n=1}^{N-1} \mathbf{K}^n
+ \mathbf{K}\sum_{n=1}^{N-1}\mathbf{A}_n,$$
$$ (\mathbf{Id-K}) \sum_{n=1}^{N-1}\mathbf{A}_n + \mathbf{A}_N \le \mathbf{A}_1 + \mathbf{A}_1 \sum_{n=1}^{N-1} \mathbf{K}^n$$
and 
$$ \sum_{n=1}^{N-1}\mathbf{A}_n \le (\mathbf{Id-K})^{-1} (\mathbf{A}_1 + \mathbf{A}_1 \sum_{n=1}^{N-1} \mathbf{K}^n).$$
Letting $N\to\infty$ we derive
$$\int_0^\infty t^p \PSI(t)\,dt= \sum_{n\ge 1} \mathbf{A}_n \le (\mathbf{Id-K})^{-1} \mathbf{A}_1 (\mathbf{Id-K})^{-1}$$
with equality if $p=1$. From \eqref{repreEN} it follows that for all $v\in[0,1]$:
\begin{equation} \label{difference}
v(\mathbf{Id-K})^{-1} \mu - \frac{1}{T} E(N_{Tv})=
\Bigpar{ v\int_{Tv}^\infty \PSI(t) dt + \frac{1}{T} \int_0^{Tv} t\, \PSI(t) dt} \mu
\end{equation}
Since $t^p \PSI(t)$ is integrable, we have
$$ T^p \int_{Tv}^\infty \PSI(t) dt \le v^{1-p} \int_{Tv}^\infty t^p \PSI(t)\,dt \to 0\;\;\text{as}\;\;T\rightarrow \infty$$
and this convergence is uniform in $v\in[0,1]$ in the case $p<1$.
It remains to prove that if $p<1$, we have 
$$\frac{1}{T^{1-p}} \int_0^T t\PSI(t)  dt \to 0\;\;\text{as}\;\;T\rightarrow \infty.$$
With $\mathbf{G}(t)=\int_0^t s^p \PSI(s) ds$, integrating by part, we obtain
$$ T^{1-p}\mathbf{G}(T)=\int_0^T t \PSI(t)\,dt + (1-p) \int_0^T t^{-p} \mathbf{G}(t) \,dt$$
and
$$ \frac{1}{T^{1-p}} \int_0^T \PSI(t) t dt = \mathbf{G}(T) - \frac{1-p}{T^{1-p}} \int_0^T t^{-p} \mathbf{G}(t) \,dt.$$
Since $\mathbf{G}(t)$ is convergent as $t\to\infty$ we finally derive that the right-hand-side in the above equality converges to $0$ as $T \rightarrow \infty$.
\end{proof}

Denote by $\norm{\cdot}$ the Euclidean norm either on $\R^d$ or on the set of $d\times d$ matrices.
\begin{lemma} \label{majM}
There exists a constant $C_{\mu,\PHI}$ such that for all $t,\De\ge 0$:
$$ E\bigpar{ {\sup_{t\le s\le t+\De}\norm{M_{s}-M_t}^2 } } \le C_{\mu,\PHI}\, \De.$$
\end{lemma}

\begin{proof}
Doob's inequality yields
$$ E\bigpar{ {\sup_{t\le s\le t+\De}\norm{M_{s}-M_t}^2 }} \le 4\sum_{i=1}^d E\bigpar{(M_{i,t+\De}-M_{i,t})^2}.$$
For each $i\in\set{1,\dots,d}$, the quadratic variation of the martingale $(M_{i,t})_{t\ge 0}$
is
$$ \bigcro{M_i,M_i}_t=\sum_{s\le t} (M_{i,s}-M_{i,s-})^2= N_{i,t}.$$
Thus we have $$E( (M_{i,t+\De}-M_{i,t})^2 ) = E(N_{i,t+\De}-N_{i,t}).$$
Besides, in view of Lemma \ref{repre} and the fact that $\int_0^\infty \PSI(t) dt = (\mathbf{Id-K})^{-1}-\mathbf{Id}$, we obtain
$$ E(N_{t+\De}-N_{t}) \le \De (\mathbf{Id-K})^{-1} \mu \quad\text{componentwise.}$$
\end{proof}

\begin{proof}[Completion of proof of Theorem \ref{LG}]
Lemma \ref{EN} with $p=0$ implies that it is enough to prove the following convergence
\begin{equation} \label{cv1}
T^{-1}\sup_{v\in[0,1]}\big\| N_{Tv} - E(N_{Tv}) \big\| \rightarrow 0\;\;\text{as}\;\;T\rightarrow \infty
\end{equation}
almost surely and in $L^2(P)$. Thanks to \eqref{repreN} of Lemma \ref{repre}, we have
\begin{align*}
\sup_{v\in[0,1]} \norm{ N_{Tv}- E(N_{Tv})} &\le \bigpar{1+\int_0^T \norm{\PSI(t)} dt} \sup_{t\le T} \norm{M_t},\\
&\le C_{\PHI}\, \sup_{t\le T} \norm{M_t}
\end{align*}
since $\PSI$ is integrable. Moreover 
$$E\bigpar{\sup_{t\le T} \norm{M_t}^2} \le C_{\mu,\PHI} T$$ 
by Lemma \ref{majM}, therefore convergence \eqref{cv1} holds in $L^2(P)$. In order to prove the almost-sure convergence, it is enough to show that
$$ T^{-1}\sup_{v\in[0,1]} \bignorm{M_{Tv}} \to 0\;\;\text{as}\;\;T\rightarrow\infty\;\;\quad\text{almost-surely.}$$
Let 
$$Z_t=(Z_{1,t},\dots,Z_{d,t})=\int_{(0,t]} \frac{1}{s+1} dM_s.$$
The quadratic variation of the martingale $Z_{i}$ satisfies 
$$[Z_i,Z_i]_t=\sum_{0<s\le t} (Z_{i,s}-Z_{i,s-})^2=\int_{(0,t]} \frac{1}{(s+1)^2} dN_{i,s}$$ and
moreover, using integration by part and \eqref{repreEN}, we have
$$ E\bigpar{\int_{(0,\infty)} \frac{1}{(s+1)^2} dN_s }
=2 E\bigpar{ \int_{(0,\infty)} \frac{N_s}{(1+s)^3} ds} < \infty.$$
Therefore $\lim_{t\to\infty} Z_t$ exists and is finite almost surely. It follows that
$$ \frac{1}{t+1} M_t= Z_t - \frac{1}{t+1}\int_0^t Z_s\, ds \rightarrow 0 \;\;\text{as}\;\;T\rightarrow\infty\;\;\text{almost surely.}$$
We deduce that almost surely, for all family $v_T\in[0,1]$, $T>0$, such that $Tv_T\to\infty$ the convergence
$ M_{Tv_T}/T \to 0$ holds. Moreover, we have $M_{Tv_T}/T\to 0$ if $\sup_T Tv_T<\infty$. In other words
$T^{-1}M_{Tv} \to 0$ uniformly uniformly in $v\in[0,1]$, almost-surely. The proof of Theorem \ref{LG} is complete.
\end{proof}

\section{Proof of Theorem \ref{TCL}} \label{proofTCL}
Let $W=(W_1,\ldots, W_d)$ be a standard $d$-dimensional Brownian motion. For $i=1,\dots,d$, put $\si_i = (\Sigma_{ii})^{1/2}$. 
\begin{lemma} \label{TCLM}
The martingales $M^{(T)}:=(T^{-1/2} M_{Tv})_{v\in[0,1]}$ converge in law for the Skorokod topology
to $(\si_1 W_1,\dots,\si_d W_d)$.
\end{lemma}
\begin{proof}
According to Theorem VIII-3.11 of \cite{JJetAS}, since the martingales $M^{(T)}$ have uniformly bounded jumps,
a necessary and sufficient condition to obtain the lemma is: for all $v\in[0,1]$, for all $1\le i<j\le d$
$$ [M^{(T)}_i, M^{(T)}_i]_v \to \si_i^2 v, \quad [M^{(T)}_i, M^{(T)}_j]_v \to 0,\;\;\text{as}\;\;T \rightarrow \infty\;\;\text{in probability.}$$
We have
\begin{equation*}
 [M^{(T)}_i, M^{(T)}_i]_v= \frac{1}{T} N_{i,Tv} \to \si_i^2 v \quad\text{in $L^2(P)$ by Theorem \ref{LG}}
 \end{equation*}
 and
 \begin{equation*}
[M^{(T)}_i, M^{(T)}_j]_v= 0
\end{equation*}
since the processes $N_i$ for $1\le i\le d$, have no common jump by construction.
 \end{proof}

\begin{proof}[Completion of proof of Theorem \ref{TCL}]
Set 
$$ X^{(T)}_v= T^{-1/2}\bigpar{ N_{Tv}-E(N_{Tv})}.$$
In view of Lemma \ref{TCLM}, it is enough to prove that
$$ \sup_{v\in[0,1]} \bignorm{ X^{(T)}_v - (\mathbf{Id-K})^{-1} M^{(T)}_v } \to 0\;\;\text{as}\;\;T\rightarrow\infty\;\;
\text{in probability.}$$
By Lemma \ref{repre}, we have
$$ X^{(T)}_v= M^{(T)}_v + \int_0^v T\PSI(Tu) M^{(T)}_{v-u}\,du,$$
hence we need to prove that
\begin{equation} \label{cvTCL}
\sup_{v\in[0,1]} \bignorm{ \int_0^v T\PSI(Tu) M^{(T)}_{v-u}\,du - \bigpar{\int_0^\infty\! \PSI(t)dt} M^{(T)}_v}
\to 0\;\;\text{as}\;\;T\rightarrow\infty
\end{equation}
in probability. We plan to use the fact that $\PSI$ is integrable and the $C$-tightness of the family $(M^{(T)})_{T>0}$. The tightness  is a consequence of Lemma \ref{TCLM} and reads:
\begin{equation} \label{CtightnessM}
\forall \ep>0 \quad \limsup_T P\Bigpar{ \sup_{\abs{u-u'}\le \eta} \bignorm{M^{(T)}_u-M^{(T)}_{u'}} >\ep}\rightarrow 0\;\;\text{as}\;\;\eta \rightarrow 0. 
\end{equation}
using also that $M^{(T)}_0=0$.
For $\eta>0$ and $v\in[0,1]$ we have
\begin{align*}
\bignorm{\int_{v\land\eta}^{v} T\PSI(Tu) M^{(T)}_{v-u}du} &\le \sup_{0\le u\le 1} \norm{M^{(T)}_u}\ \int_\eta^1 T\norm{\PSI(Tu)}du \\
&\le \sup_{0\le u\le 1} \norm{M^{(T)}_u}\ \int_{T\eta}^\infty\norm{\PSI(t)}dt \rightarrow 0
\end{align*}
as $T\rightarrow \infty$ in probability, since $\sup_{0\le u\le 1} \|M^{(T)}_u\|$ is bounded in probability and $\int_{T\eta}^\infty\norm{\PSI(t)}dt\to 0$ as $T \rightarrow \infty$. Moreover
$$ \bignorm{\int_0^{v\land\eta} T\PSI(Tu) \bigpar{M^{(T)}_v-M^{(T)}_{v-u}}du} \le \sup_{\abs{u-u'}\le \eta} \bignorm{M^{(T)}_u-M^{(T)}_{u'}}\ \int_0^\infty \norm{\PSI(t)}dt,$$
therefore, in order to prove \eqref{cvTCL} it suffices to show that for all $\eta>0$, the convergence
$$\sup_{v\in[0,1]} \bignorm{ \bigpar{\int_0^\infty \PSI(t)dt - \int_0^{v\land\eta} T\PSI(Tu) du}M^{(T)}_v }\rightarrow 0\;\;\text{as}\;\;T\rightarrow \infty$$
holds in probability. It readily follows from \eqref{CtightnessM} and the upper-bound 
\begin{align*}
& \bignorm{\Bigpar{\int_0^\infty \PSI(t)dt - \int_0^{v\land\eta} T\PSI(Tu) du}M^{(T)}_v} \\
\le & 
\begin{cases}
\displaystyle \int_{T\de}^\infty \norm{\PSI(t)} dt\  \sup_{u} \norm{M^{(T)}_u} & \text{if $v>\de$}\\ \\
\displaystyle \int_0^\infty \norm{\PSI(t)}dt\ \sup_{u\le \de} \norm{M^{(T)}_u}  & \text{if $v\le \de,$}
\end{cases}
\end{align*}
with $0<\de<\eta$.
\end{proof}

\begin{proof}[Proof of Corollary \ref{TCLbis}] 
By \eqref{uibis}, Lemma \ref{EN} with $p=1/2$ yields
$$ T^{1/2} \bigpar{ T^{-1}E(N_{Tv})- v (\mathbf{Id-K})^{-1}\mu } \to 0\;\;\text{as}\;\;T\rightarrow\infty$$
uniformly in $v\in[0,1]$.
Moreover, by \eqref{sta}, Theorem \ref{TCL} yields
$$T^{1/2} \bigpar{ T^{-1}N_{Tv}-T^{-1}E(N_{Tv}) }\rightarrow (\mathbf{Id-K})^{-1} \mathbf{\Sigma}^{1/2} W$$
in distribution as $T\rightarrow \infty$ and the result follows.
\end{proof}
\section{Proof of Theorem \ref{SP}} \label{proofSP}
Set
\begin{equation} \label{defY}
Y_t=\int_0^t \PSI(t-s) M_s \, ds
\end{equation}
in order that $X=M+Y$, see Lemma \ref{repre}. For all $0\le \ep\le \eta\le 1$, for all integer $1\le k_0\le T/\De$, define
$${\mathcal D}_{k_0,\varepsilon,\eta}(X)_{T,\Delta}=\frac{1}{T}\sum_{k=k_0}^{\PE{T/\De}} \bigpar{X_{(k-1)\De+\eta\De}-X_{(k-1)\De+\ep\De}}.$$
\begin{lemma} \label{VN}
There exists a function $T\leadsto \xi_{\mu,\PHI}(T)$ such that
$\xi_{\mu,\PHI}(T)\to0 $ as $T\to\infty$ and 
such that for all $0\le \ep\le \eta\le 1$ and all integer $1\le k_0\le T/\De$, we have
$$
E\big(\big\|(1-\frac{k_0\De}{T})(\eta-\ep)(\mathbf{Id-K})^{-1}\mu 
-{\mathcal D}_{k_0,\varepsilon,\eta}(N)_{T,\Delta}\big\|^2\big)  \le \xi_{\mu,\PHI}(T).$$
\end{lemma}
\begin{proof}
First we prove that
\begin{equation} \label{aux}
 \E\big(\big\|{\mathcal D}_{k_0,\varepsilon,\eta}(X)_{T,\Delta}\|^2\big) \leq C_{\mu,\PHI}T^{-1}
 \end{equation}
for some constant $C_{\mu,\PHI}$ that depends on $\mu$ and $\PHI$ only.
Using 
\begin{align*}
X_{(k-1)\De+\eta\De}-X_{(k-1)\De+\ep\De}&
=M_{(k-1)\De+\eta\De}-M_{(k-1)\De+\ep\De} \\
+\int_0^\infty ds\, &\PSI(s) \bigpar{M_{(k-1)\De+\eta\De-s}-M_{(k-1)\De+\ep\De-s}}
\end{align*}
and the fact that $\PSI$ is integrable, it suffices to prove that
$$E\big(\big\|{\mathcal D}_{k_0,\varepsilon,\eta}(M_{\cdot-s})_{T,\Delta}\|^2\big) \leq C_{\mu,\PHI}T^{-1}.$$
Set
$$ S_n=\sum_{k=1}^n \bigpar{ M_{(k-1)\De+\eta\De-s}-M_{(k-1)\De+\ep\De-s}}.$$
Clearly $(S_n)_{n \geq 1}$ is a discrete martingale thus
$$ E\bigpar{ \bignorm{S_{\PE{T\De}}-S_{k_0-1}}^2 } 
= \sum_{k=k_0}^{\PE{T\De}}
E\bigpar{\bignorm{ M_{(k-1)\De+\eta\De-s}-M_{(k-1)\De+\ep\De-s}}^2}.$$
By Lemma \ref{majM} we have
$E\bigpar{ \bignorm{S_{\PE{T\De}}}^2 } \le C_{\mu,\PHI} (\eta-\ep) T$
and \eqref{aux} follows. It remains to prove that
$$\big\|(1-\frac{k_0\De}{T})(\eta-\ep)(\mathbf{Id-K})^{-1}\mu 
-E\big({\mathcal D}_{k_0,\varepsilon,\eta}(N)_{T,\Delta}\big)\big\| \le \tilde\xi_{\mu,\PHI}(T)$$
where $\tilde\xi_{\mu,\PHI}(T) \to 0$ as $T\rightarrow \infty$. By Lemma \ref{repre}, we have
\begin{align*}
 & E\bigpar{ N_{(k-1)\De+\eta\De}-N_{(k-1)\De+\ep\De}} \\
 =\; & 
(\eta-\ep)\De \Bigpar{\mu+\int_0^{(k-1)\De+\ep\De} dr \PSI(r) \mu }\\
&+ \int_{(k-1)\De+\ep\De}^{(k-1)\De+\eta\De} dr \PSI(r) ((k-1)\De+\eta\De-r) \mu\\
=\; &(\eta-\ep)\De (\mathbf{Id-K})^{-1}\mu
- (\eta-\ep)\De \int_{(k-1)\De+\ep\De}^\infty dr \PSI(r) \\
 &+\int_{(k-1)\De+\ep\De}^{(k-1)\De+\eta\De} dr \PSI(r) \big((k-1)\De+\eta\De-r\big) \mu.
\end{align*}
Finally
$$\int_{(k-1)\De+\ep\De}^{(k-1)\De+\eta\De} dr \PSI(r) ((k-1)\De+\eta\De-r)
\le (\eta-\ep)\De \int_{(k-1)\De+\ep\De}^\infty dr \PSI(r).$$
We conclude noting that
\begin{align*}
\frac{\De}{T} \sum_{k=k_0}^{\PE{T/\De}} \int_{(k-1)\De+\ep\De}^\infty dr \PSI(r)&=
\frac{\De}{T} \int_{(k_0-1)\De+\ep\De}^\infty dr \PSI(r) \sum_{k=k_0}^{\PE{T/\De}} \ind{\set{(k-1)\De+\ep\De <r}}\\
&\le \frac{1}{T} \int_0^\infty dr\, \PSI(r) (r\land T), 
\end{align*}
This last quantity converges to $0$ as $T \rightarrow \infty$ by an argument similar to the end of proof of Lemma, \ref{EN}, using that $\PSI$ is integrable.
\end{proof}


\begin{lemma} \label{bdg}
There exists a constant $C_{\mu,\PHI}$ such that for all $t,h\ge 0$, we have
$$ E\bigpar{ {\sup_{t\le s\le t+h}\norm{M_{s}-M_t}^4 } } \le C_{\mu,\PHI}\, (h+h^2)$$
\end{lemma}
\begin{proof}
According to the Burkholder-Davis-Gundy inequality, we have
\begin{align*}
E\bigpar{ {\sup_{t\le s\le t+h}\norm{M_{s}-M_t}^4 } \mid \F_t} &\le C \sum_{i=1}^d E\bigpar{ \bigpar{[M_i,M_i]_{t+h}-[M_i,M_i]_t}^2  \mid \F_t}\\
&=C \sum_{i=1}^d E\bigpar{ \bigpar{N_{i,t+h}-N_{i,t}}^2  \mid \F_t},
\end{align*}
hence
$$E\bigpar{ {\sup_{t\le s\le t+h}\norm{M_{s}-M_t}^4 } }
\le C \bigpar{ E\bigpar{ \bignorm{X_{t+h}-X_{t}}^2 } + \bignorm{E\bigpar{N_{t+h}-N_t}}^2¬†}.$$
By Lemma \ref{repre}, we have
$$ E\bigpar{N_{t+h}-N_t} \le h(\mathbf{Id-K})^{-1} \mu \quad\text{componentwise}$$
and
$$ \Bigpar{ E\bigpar{ \bignorm{X_{t+h}-X_{t}}^2 }}^{1/2}
\le \int_0^\infty \norm{\PSI(s)} \Bigpar{E\bigpar{\bignorm{M_{t+h-s}-M_{t}}^2}}^{1/2} \, ds.$$
The conclusion follows using Lemma \ref{majM} and the fact that $\PSI$ is integrable.
\end{proof}

With the notation introduced in Section \ref{SP}, the quantity $\mathbf{V}_{\De,T}(M_{\cdot-s},M_{\cdot-t+\tau})$ is equal to
$$\frac{1}{T} \sum_{k=1}^{\PE{T/\De}} \bigpar{M_{k\De-s}-M_{(k-1)\De-s}}  \bigpar{M_{k\De-t+\tau}-M_{(k-1)\De-t+\tau}}^\T.$$
\begin{lemma} \label{VM}
For all $s,t\ge 0$ we have
$$\mathbf{V}_{\De_T,T}(M_{\cdot-s},M_{\cdot-t+\tau_T})- \bigpar{1-\tfrac{\abs{t-s-\tau_T}}{\De_T}}^+\, \mathbf{\Sigma} \to 0\;\;\text{as}\;\;T\rightarrow \infty
\;\;\text{in}\;\;L^2(P)$$
and
$$ E\bigpar{ \bignorm{\mathbf{V}_{\De_T,T}(M_{\cdot-s},M_{\cdot-t+\tau_T})}^2 } \le C_{\mu,\PHI}.$$
\end{lemma}

\begin{proof}
We start with proving preliminary estimates. Let $b_1,b_2,b'_1,b'_2$ be real numbers such that $b_1\le b_2$, $b'_1\le b'_2$ and
$(b_1,b_2)\cap(b'_1,b'_2)=\emptyset$. Using that $M$ is a martingale, we successively obtain
\begin{align*} 
&E\big(\bignorm{T^{-1}\sum_{k=1}^{\PE{T/\De}} \bigpar{M_{(k-1)\De+b_2}-M_{(k-1)\De+b_1}}  \bigpar{M_{(k-1)\De+b'_2}-M_{(k-1)\De+b'_1}}^\T}^2\big)\\
=\;&\frac{1}{T^2} \sum_{k=1}^{\PE{T/\De}} E\big(\bignorm{\bigpar{M_{(k-1)\De+b_2}-M_{(k-1)\De+b_1}}  \bigpar{M_{(k-1)\De+b'_2}-M_{(k-1)\De+b'_1}}^\T}^2\big) \nonumber\\ 
=\;&\frac{1}{T^2} \sum_{k=1}^{\PE{T/\De}} E\big(\bignorm{M_{(k-1)\De+b_2}-M_{(k-1)\De+b_1}}^2 \bignorm{  \bigpar{M_{(k-1)\De+b'_2}-M_{(k-1)\De+b'1}}^\T}^2\big).
\end{align*}
By Cauchy-Schwarz, this last quantity is less than 
$$\frac{1}{T^2}\sum_{k=1}^{\PE{T/\De}} \big( E(\norm{M_{(k-1)\De+b_2}-M_{(k-1)\De+b_1}}^4)E(\bignorm{  \bigpar{M_{(k-1)\De+b'_2}-M_{(k-1)\De+b'_1}}^\T}^4) \big)^{\frac{1}{2}}$$
which in turn, using  Lemma \ref{bdg}, is bounded by
\begin{equation} \label{disjoint}
C_{\mu,\PHI} \frac{1}{T\De}  \bigpar{b_2-b_1+ (b_2-b_1)^2}^{1/2}
\bigpar{b'_2-b'_1+ (b'_2-b'_1)^2}^{1/2}.
\end{equation}
Moreover, if $b_2-b_1 \le \De$, using that $[M,M]= \mathbf{diag}(N)$ and Lemma \ref{bdg}, we obtain that
\begin{align*} 
&E\Big(\Big\|\frac{1}{T} \sum_{k=1}^{\PE{T/\De}} \mathbf{diag}\bigpar{ N_{(k-1)\De+b_2}- N_{(k-1)\De_T+b_1} }\\
- &\frac{1}{T} \sum_{k=1}^{\PE{T/\De}}  \bigpar{M_{(k-1)\De+b_2}-M_{(k-1)\De+b_1}}  \bigpar{M_{(k-1)\De+b_2}-M_{(k-1)\De+b_1}}^\T\Big\|^2\Big)
\end{align*}
is less than
\begin{equation} \label{same}
C_{\mu,\PHI} \frac{1}{T\De}  \bigpar{b_2-b_1+ (b_2-b_1)^2}.
\end{equation}
We are ready to prove Lemma \ref{VM}. It is a consequence of Lemma \ref{VN} and the fact that there exists $a_2=a_2(s,t,\De_T,\tau_T)\le a_3=a_3(s,t,\De_T,\tau_T)$ such that
$$a_3-a_2= \bigpar{ \De_T- \abs{t-s-\tau_T}}^+, \quad
\frac{a_2}{T} \rightarrow 0\;\;\text{as}\;\;T\rightarrow \infty$$
holds, together with the estimate
\begin{equation} \label{aux2}
\begin{split}
&E\big(\bignorm{ \frac{1}{T} \sum_{k=1}^{\PE{T/\De_T}} \mathbf{diag}\bigpar{ N_{(k-1)\De_T+a_3}- N_{(k-1)\De_T+a_2} }
- \mathbf{V}_{\De_T,T}(M_{\cdot-s},M_{\cdot-t+\tau_T}) }^2\big) \\
&\le C_{\mu,\PHI} \frac{1}{T}  \bigpar{1+\De_T}.
\end{split}
\end{equation}
Indeed, if $t-\tau_T+\De_T\le s$ or $s\le t-\tau_T-\De_T$ then the upper bound we obtained in \eqref{disjoint} entails
$$E\big(\bignorm{\mathbf{V}_{\De_T,T}(M_{\cdot-s},M_{\cdot-t+\tau_T})}^2\big) \le C_{\mu,\PHI} \frac{1}{T}  \bigpar{1+\De_T}.$$
Let us first consider the case $t-\tau_T\le s \le t-\tau_T+\De_T$. Set
$$ a_1:=s \le a_2:=-t+\tau_T \le a_3:=\De_T-s \le a_4:=\De_T-t+\tau_T.$$
We use the following decomposition
\begin{align*}
&\bigpar{M_{k\De_T-s}-M_{(k-1)\De_T-s}}  \bigpar{M_{k\De_T-t+\tau_T}-M_{(k-1)\De_T-t+\tau_T}}^\T
\\
=\,&\bigpar{M_{(k-1)\De_T+a_3}-M_{(k-1)\De_T+a_2}}  \bigpar{M_{(k-1)\De_T+a_3}-M_{(k-1)\De_T+a_2}}^\T \\
+& \bigpar{M_{(k-1)\De_T+a_3}-M_{(k-1)\De_T+a_1}}  \bigpar{M_{(k-1)\De_T+a_4}-M_{(k-1)\De_T+a_3}}^\T\\
 +&\bigpar{M_{(k-1)\De_T+a_2}-M_{(k-1)\De_T+a_1}}  \bigpar{M_{(k-1)\De_T+a_3}-M_{(k-1)\De_T+a_2}}^\T.
\end{align*}
On the one hand, \eqref{disjoint} readily yields that both
$$
 E\big(\bignorm{ \frac{1}{T} \sum_{k=1}^{\PE{T/\De_T}} \bigpar{M_{(k-1)\De_T+a_3}-M_{(k-1)\De_T+a_1}}  \bigpar{M_{(k-1)\De_T+a_4}-M_{(k-1)\De_T+a_3}}^\T}^2\big) 
$$
and
$$
E\big(\bignorm{ \frac{1}{T} \sum_{k=1}^{\PE{T/\De_T}}  \bigpar{M_{(k-1)\De_T+a_2}-M_{(k-1)\De_T+a_1}}  \bigpar{M_{(k-1)\De_T+a_3}-M_{(k-1)\De_T+a_2}}^\T }^2\big)
$$
are less than $C_{\mu,\PHI} (1+\De_T)/T$.
On the other hand, by \eqref{same}, the same estimate holds for
\begin{multline*}
E\big(\bignorm{ \frac{1}{T} \sum_{k=1}^{\PE{T/\De_T}} \mathbf{diag}\bigpar{ N_{(k-1)\De_T+a_3}- N_{(k-1)\De_T+a_2} }\\
- \frac{1}{T} \sum_{k=1}^{\PE{T/\De_T}}  \bigpar{M_{(k-1)\De_T+a_3}-M_{(k-1)\De_T+a_2}}  \bigpar{M_{(k-1)\De_T+a_3}-M_{(k-1)\De_T+a_2}}^\T}^2\big),
\end{multline*}
therefore \eqref{aux2} holds in that case. If now $t-\tau_T-\De_T\le s\le t-\tau_T$, setting
$$ a_1:=-t+\tau_T\le a_2:=-s \le a_3:=\De_T-t+\tau_T \le a_4:=\De_T-s,$$
one readily checks that \eqref{aux2} holds using the same arguments.
\end{proof}

\begin{proof}[Completion of proof of Theorem \ref{SP}]
Since $X=M+Y$ by \eqref{defY}, we have the following decomposition
\begin{align*}
& \mathbf{V}_{\De,T}(X,X_{\tau+\cdot}) \\
=\; & \mathbf{V}_{\De,T}(M,M_{\tau+\cdot}) + \mathbf{V}_{\De,T}(Y,Y_{\tau+\cdot})
+\mathbf{V}_{\De,T}(M,Y_{\tau+\cdot})+\mathbf{V}_{\De,T}(Y,M_{\tau+\cdot}).
\end{align*}
Setting $M_t=0$ for $t\le 0$, we can write
$Y_t=\int_0^\infty \PSI(s) M_{t-s}\,ds$,
therefore
\begin{align*}
& (Y_{i\De}-Y_{(i-1)\De})(Y_{i\De+\tau}-Y_{(i-1)\De+\tau})^\T\\
=&\;\int_{\R_+^2} \!\! ds\, dt\, \PSI(s) \bigpar{M_{i\De-s}-M_{(i-1)\De-s}}  \bigpar{M_{i\De-t+\tau}-M_{(i-1)\De-t+\tau}}^\T \PSI(t)^\T,
\end{align*}
hence
$$ \mathbf{V}_{\De,T}(Y,Y_{\tau+\cdot}) =\int_{\R_+^2} \!\! ds\, dt\, \PSI(s) \mathbf{V}_{\De,T}(M_{\cdot-s},M_{\cdot-t+\tau})
\PSI(t)^\T.$$
Likewise 
$$ \mathbf{V}_{\De,T}(M,Y_{\tau+\cdot})=\int_0^\infty dt\, \mathbf{V}_{\De,T}(M,M_{\cdot-t+\tau}) \PSI(t)^\T,$$
$$\mathbf{V}_{\De,T}(Y,M_{\tau+\cdot})=\int_0^\infty dt\, \PSI(t) \mathbf{V}_{\De,T}(M_{\cdot-t},M_{\cdot+\tau}).$$
In view of Lemma \ref{VM} and the fact that $\PSI$ is integrable, by Lebesgue dominated convergence theorem, we successively obtain
$$ \mathbf{V}_{\De_T,T}(Y,Y_{\tau_T+\cdot}) - \int_{\R_+^2} \!\! ds\, dt\, \PSI(s)
\bigpar{1-\abs{t-s-\tau_T}/\De_T}^+\, \mathbf{\Sigma}\, \PSI(t)^\T \to 0,$$
$$ \mathbf{V}_{\De_T,T}(M,Y_{\tau_T+\cdot})- \int_0^\infty dt\,\bigpar{1-\abs{t-\tau_T}/\De_T}^+\, \mathbf{\Sigma}\,  \PSI(t)^\T \to 0,$$
and
$$\mathbf{V}_{\De_T,T}(Y,M_{\tau_T+\cdot})-\int_0^\infty dt\, \PSI(t) \bigpar{1-\abs{t+\tau_T}/\De_T}^+\, \mathbf{\Sigma} \to 0$$
in $L^2(P)$ as $T \rightarrow \infty$. The proof is complete.
\end{proof}

\begin{proof}[Proof of Corollary \ref{SPbis}]
In view of Theorem \ref{SP}, we need to show 
$$\mathbf{V}_{\De_T,T}(X,X_{\tau_T+\cdot}) - \mathbf{V}_{\De_T,T}(\tilde X,\tilde X_{\tau_T+\cdot}) \to 0\;\;\text{as}\;\;T\rightarrow \infty\;\;\text{in}\;\;L^2(P).$$
By Cauchy-Schwarz inequality, this convergence is a consequence of 
$$\sup_{\tau\in\R} \frac{1}{T} \sum_{k=1}^{\PE{T/\De_T}} \Bigpar{\eta_i(\tau+k\De_T)-\eta_i(\tau+(k-1)\De_T)}^2 \to 0,
\quad 1\le i\le d,$$
where $\eta(t)=\bigpar{t (\mathbf{Id-K})^{-1} \mu -  E(N_t)}\, \ind{\set{t\ge 0}}$.
By \eqref{difference}, for $t\ge 0$, we have the decomposition
$$\eta(t)=\Bigpar{ t\int_t^\infty \PSI(s) ds + \int_0^t \PSI(s) s ds} \mu$$
therefore the function $\eta$ is absolutely continuous and we have
$$ \eta'(t)=\int_t^\infty \PSI(s)ds \, \mu \,\ind{\set{t\ge 0}}.$$
We derive
$$\frac{1}{T} \sum_{k=1}^{\PE{T/\De_T}} \Bigpar{\eta_i(\tau+k\De_T)-\eta_i(\tau+(k-1)\De_T)}^2
\le \frac{\De_T}{T} \int_0^\infty \eta'_i(s)^2 ds. $$
It remains is to prove 
$$
\int_0^\infty \norm{\eta'(t)}^2 dt <\infty.
$$
We have
$$ \eta'(t) \le t^{-1/2} \int_t^\infty s^{1/2} \PSI(s)\,ds \mu\;\;\text{and}\;\;\norm{\eta'(t)} \le C_{\mu,\PHI} t^{-1/2}$$
since $s\leadsto s^{1/2} \PSI(s)$ is integrable. Finally
\begin{align*}
\int_0^\infty \norm{\eta'(t)}^2 dt &  \leq  C_{\mu,\PHI} \int_0^\infty dt\, t^{-1/2} \int_t^\infty ds \norm{\PSI(s)} \\
& = 2\,C_{\mu,\PHI} \int_0^\infty ds \norm{\PSI(s)} s^{1/2}<\infty
\end{align*}
and the result follows.
\end{proof}

\section*{Appendix} 
\subsection*{Proof of Proposition \ref{ex4prop1}}
The spectral radius of
$\mathbf{K}=\left(\begin{smallmatrix}
0 & 0 & \smallint h & 0\\
0 & 0 & 0 & \smallint h\\
\smallint g & 0 & 0 & 0\\
0 & \smallint g & 0 & 0
\end{smallmatrix}\right)$
is equal to $\smallint g \smallint h$, and we have
$$ \bigpar{\mathbf{Id-K}}^{-1}
=\frac{1}{1-\smallint h\smallint g} \left(\begin{smallmatrix}
1 & 0 & \smallint h & 0\\
0 & 1 & 0 & \smallint h\\
\smallint g & 0 & 1 & 0\\
0 & \smallint g & 0 & 1
\end{smallmatrix}\right),$$
therefore
$$
\bigpar{\mathbf{Id-K}}^{-1} \mu=\frac{1}{1-\smallint h\smallint g}
\left(\begin{smallmatrix}
\nu_1 \\ \nu_1 \\ \nu_2 \\ \nu_2
\end{smallmatrix}\right)$$
where $\nu_1$ and $\nu_2$ are given by \eqref{defnu}. Set $X=N-E(N)$. 
By symmetry, we have $N=(N_2,N_1,N_4,N_3)$ in distribution, thus $E(N_{1,t})=E(N_{2,t})$ and $E(N_{3,t})=E(N_{4,t})$ for all $t$. Consequently
\begin{equation} \label{SX}
S_1=X_1-X_2 \quad\text{and}\quad S_2=X_3-X_4.
\end{equation}
According to Theorem \ref{TCL} the processes
$T^{-1/2}X_{Tv}$ converge in law to the process
$Y_v= (\mathbf{Id-K})^{-1} \mathbf{\Sigma}^{1/2} W_v$ with
$$\mathbf{\Sigma}=\frac{1}{1-\smallint h\smallint g}\left(\begin{smallmatrix}
\nu_1 & 0 & 0 & 0\\
0 & \nu_1 & 0 & 0\\
0 & 0 & \nu_2 & 0\\
0 & 0 & 0 & \nu_2
\end{smallmatrix}\right).$$
Therefore the processes $T^{-1/2}\bigpar{S_{1,Tv},S_{2,Tv}}_{v\in[0,1]}$
converge in distribution to $\bigpar{Y_1-Y_2,Y_3-Y_4}$ and Proposition \ref{ex4prop1} is proved.

\subsection*{Proof of Proposition \ref{ex4prop2}}

From \eqref{SX} it follows that
$$ V_{\De_T,T}(S_1,S_{1,\tau_T+\cdot})=
\left(\begin{smallmatrix} 1 & -1 & 0 & 0 \end{smallmatrix}\right)
V_{\De_T,T}(X,X_{\tau_T+\cdot})
\left(\begin{smallmatrix} 1 \\ -1 \\ 0 \\ 0 \end{smallmatrix}\right)$$
and
$$ V_{\De_T,T}(S_1,S_{2,\tau_T+\cdot})=
\left(\begin{smallmatrix} 1 & -1 & 0 & 0 \end{smallmatrix}\right)
V_{\De_T,T}(X,X_{\tau_T+\cdot})
\left(\begin{smallmatrix} 0 \\ 0 \\ 1 \\ -1 \end{smallmatrix}\right).$$
Consequently Proposition \ref{ex4prop2} follows from Theorem \ref{SP} with
$$ C_{11}(\De,\tau)=
\left(\begin{smallmatrix} 1 & -1 & 0 & 0 \end{smallmatrix}\right) \mathbf{v}_{\De,\de}
\left(\begin{smallmatrix} 1 \\ -1 \\ 0 \\ 0 \end{smallmatrix}\right),\ 
C_{12}(\De,\tau)=
\left(\begin{smallmatrix} 1 & -1 & 0 & 0 \end{smallmatrix}\right) \mathbf{v}_{\De,\de}
\left(\begin{smallmatrix} 0 \\ 0 \\ 1 \\ -1 \end{smallmatrix}\right).$$
It remains to compute $C_{11}$ and $C_{12}$.
First we compute $\PSI = \sum_{n\ge 1} \PHI_n$.
We readily check that for all $s,t\ge 0$:
$\PHI(t-s)\PHI(s)=h(t-s)g(s)\mathbf{Id}$. Thus $\PHI_2= (h\star g) \mathbf{Id}$.
We derive
$$ \PHI_{2n}=(h\star g)^{\star n} \mathbf{Id}, \quad \PHI_{2n+1}=
\left(\begin{smallmatrix}
0 & 0 & (h\star g)^{\star n}\star h & 0\\
0 & 0 & 0 & (h\star g)^{\star n}\star h \\
(h\star g)^{\star n}\star g & 0 & 0 & 0\\
0 & (h\star g)^{\star n}\star g& 0 & 0
\end{smallmatrix}\right).$$
and we obtain
\begin{align*}
\PSI =&\sum_{n\ge 1} \PHI_{2n} + \sum_{n\ge 0} \PHI_{2n+1} \\
=& \begin{pmatrix}
F & 0 & (\de_0+F)\star h & 0\\
0 & F & 0 & (\de_0+F)\star h\\
(\de_0+F)\star g & 0 & F & 0\\
0 & (\de_0+F)\star g & 0 & F
\end{pmatrix}
\end{align*}
where $F=\sum_{n\ge 1} (h\star g)^{\star n}$. Set $\tilde F(ds) = \de_0(ds) + F(t)ds$. Standard computations yield
$$\mathbf{Id}\de_0(ds) + \PSI(s)ds=
\begin{pmatrix}
\tilde F(ds) & 0 & \tilde F\star h(s)ds & 0\\
0 & \tilde F(ds) & 0 & \tilde F\star h(s)ds\\
\tilde F\star g(s)ds & 0 & \tilde F(ds) & 0\\
0 & \tilde F\star g(s)ds & 0 & \tilde F(ds)
\end{pmatrix}
$$
and
\begin{multline*}
\bigpar{\mathbf{Id}\de_0(ds) + \PSI(s)ds }\mathbf{\Sigma}\bigpar{\mathbf{Id}\de_0(dt) + \PSI(t)^*dt}
=\\
\frac{1}{1-\smallint h\smallint g}
\begin{pmatrix}
a_{11}(ds,dt) & 0 & a_{13}(ds,dt) & 0\\
0 & a_{11}(ds,dt) & 0 & a_{13}(ds,dt)\\
a_{31}(ds,dt) & 0 & a_{33}(ds,dt) & 0\\
0 & a_{31}(ds,dt) & 0 & a_{33}(ds,dt)
\end{pmatrix}
\end{multline*}
with:
$$ a_{11}(ds,dt)=\nu_1\tilde F(ds)\tilde F(dt)+ \nu_2\tilde F\star h(ds)\tilde F\star h(dt),$$
$$ a_{13}(ds,dt)=\nu_1\tilde F(ds) \tilde F\star g(dt)+
\nu_2 \tilde F\star h(ds)\tilde F(dt),$$
$$a_{31}(ds,dt)=\nu_2\tilde F(ds) \tilde F\star h(dt)+
\nu_1 \tilde F\star g(ds)\tilde F(dt),$$
$$ a_{33}(ds,dt)=\nu_2\tilde F(ds)\tilde F(dt)+ \nu_1\tilde F\star g(ds)\tilde F\star g(dt).$$
Therefore
$$
C_{11}(\De,\tau)=\int_{[0,\infty)^2} \ga_\De(t-s-\tau) 2 a_{11}(ds,dt),
$$
and
$$
C_{12}(\De,\tau)=\int_{[0,\infty)^2} \ga_\De(t-s-\tau) 2 a_{31}(ds,dt).
$$
To complete the proof of Proposition \ref{ex4prop2}, it suffices to use that for two finite measures $\mu$ and $\nu$ on $\R$ one has (for all $\De>0$, $\tau\in\R$)
$$ \int_{\R_+^2} \ga_\De(t-s-\tau) \mu(ds)\nu(dt)
=\ga_\De\star\nu\star\check \mu (\tau)$$
where $\check \mu$ is the image of $\mu$ by $x\leadsto -x$.

\section*{Acknowledgement} The research of E. Bacry and J.F. Muzy is supported in part by the {\it Chair Financial Risks of the Risk Foundation}. The research of M. Hoffmann is supported in part by the {\it Agence Nationale de la Recherche}, Grant No. ANR-08-BLAN-0220-01.
\bibliographystyle{plain}       
\bibliography{biblio}           
\end{document}